\documentclass[fleqn,final,3p,11pt]{elsarticle}

\usepackage{bbm}
\usepackage{lineno,hyperref}
\hypersetup{pdfauthor=author}
\usepackage{amsmath,amsthm,amssymb}
\usepackage{dirtytalk}
\usepackage{comment}

\usepackage{siunitx}
\usepackage{xcolor}
\usepackage{booktabs,colortbl, array}

\definecolor{rulecolor}{RGB}{0,71,171}
\definecolor{tableheadcolor}{gray}{0.92}
\newcommand{\topline}{ %
        \arrayrulecolor{rulecolor}\specialrule{0.1em}{\abovetopsep}{0pt}%
        \arrayrulecolor{tableheadcolor}\specialrule{\belowrulesep}{0pt}{0pt}%
        \arrayrulecolor{rulecolor}}
\newcommand{\midtopline}{ %
        \arrayrulecolor{tableheadcolor}\specialrule{\aboverulesep}{0pt}{0pt}%
        \arrayrulecolor{rulecolor}\specialrule{\lightrulewidth}{0pt}{0pt}%
        \arrayrulecolor{white}\specialrule{\belowrulesep}{0pt}{0pt}%
        \arrayrulecolor{rulecolor}}
\newcommand{\bottomline}{ %
        \arrayrulecolor{white}\specialrule{\aboverulesep}{0pt}{0pt}%
        \arrayrulecolor{rulecolor} %
        \specialrule{\heavyrulewidth}{0pt}{\belowbottomsep}}%

\usepackage{rotating}

\usepackage{graphicx}
\usepackage{multirow}
\usepackage{colortbl}%
  \newcommand{\myrowcolour}{\rowcolor{white}}
\usepackage{booktabs}
\definecolor{Gray}{gray}{0.925}
\usepackage{enumitem,booktabs,cfr-lm}
\usepackage[referable]{threeparttablex}
\renewlist{tablenotes}{enumerate}{1}
\makeatletter
\setlist[tablenotes]{label=\tnote{\alph*},ref=\alph*,itemsep=\z@,topsep=\z@skip,partopsep=\z@skip,parsep=\z@,itemindent=\z@,labelindent=\tabcolsep,labelsep=.2em,leftmargin=*,align=left,before={\footnotesize}}
\makeatother

\usepackage{placeins,diagbox}
\usepackage{booktabs} 
\usepackage{caption} 
\usepackage{subcaption} 

\usepackage{graphicx}
\usepackage[all]{nowidow}
\usepackage[utf8]{inputenc}
\usepackage{tikz}
\usetikzlibrary{er,positioning,bayesnet}
\usepackage{multicol}
\usepackage{algpseudocode,algorithm,algorithmicx}

 \newtheorem{thm}{Theorem}[section]

\newtheorem{defn}[thm]{Definition}
 \newtheorem{rem}[thm]{Remark}
 \numberwithin{equation}{section}
\modulolinenumbers[5]
\begin{document}

\begin{frontmatter}

\cortext[cor1]{Corresponding author. Tel: +16626940182 \\ \indent\ \ \ E-mail addresses: 
dd1424@msstate.edu (D. Damircheli),
smm.kazemi@khu.ac.ir (S.-M.-M. Kazemi), bastani@iasbs.ac.ir (A. Foroush Bastani).}

\title{ On Meshfree Collocation to Compute the Probability of Default under a Regime-Switching Synchronous-Jump Tempered Stable L\'{e}vy Model}

\author[label1]{Davood Damircheli}
\author[label1]{{Mohsen Razzaghi}\corref{cor1}}
\author[label2]{{Seyed-Mohammad-Mahdi Kazemi}}
\author[label3]{{Ali Foroush Bastani}}
\address[label1]{Department and Organization, Mississippi State University,\\ P.O. Box 39759, Starkville, USA}
\address[label2]{Department of Financial Mathematics, Faculty of Financial Sciences, Kharazmi University,\\ P.O. Box 15936-56311, Tehran, Iran}
\address[label3]{Department of Mathematics, Institute for Advanced Studies in Basic Sciences,\\
P.O. Box 45195-1159, Zanjan, Iran}
\vspace{0 cm}

\begin{abstract}
In the paper [Hainaut, D. and Colwell, D.B., {\rm A structural model for credit risk with switching processes and synchronous jumps}, The European Journal of Finance 22(11) (2016): 1040-1062], the authors exploit a synchronous-jump regime-switching model to compute the default probability of a publicly traded company. Here, we first generalize the proposed L\'{e}vy model to more general setting of tempered stable processes recently introduced into the finance literature. Based on the singularity of the resulting partial integro-differential operator, we propose a general framework based on strictly positive-definite functions to de-singularize the operator. We then analyze an efficient meshfree collocation method based on radial basis functions to approximate the solution of the corresponding system of partial integro-differential equations arising from the structural credit risk model. We show that under some regularity assumptions, our proposed method naturally de-sinularizes the problem in the tempered stable case. Numerical results of applying the method on some standard examples from the literature confirms the accuracy of our theoretical results and numerical algorithm.

\end{abstract}

\end{frontmatter}

\section{Introduction}
In recent years, credit risk models have become an indispensable tool for regulators to assess the performance of financial institutions, specially banking and credit providing entities. Motivated by the proposals in Basel accords\footnote{\url{https://www.bis.org/bcbs/publ/d424.htm}}, accurate credit risk models are now essential tools in predicting the default risk of loan portfolios, evaluating the vulnerability of lender institutions and estimating reliable economic capital levels for banks to remain solvent at a given confidence level and time horizon \cite{mcneil2015quantitative}. They also provide useful tools for market participants to identify, measure,
monitor and control their exposure to credit risk \cite{fong2006credit}.

The two primary approaches in the literature to model the default risk of debt obligations and to price credit risky securities are structural and reduced-form (a.k.a intensity or statistical) models \cite{duffie2012credit}. While the ``reduced-form'' approach, abstracts away from the economic notion of bankruptcy and treats the default event as an exogenous occurrence governed by a specific jump-diffusion process, the structural models use the evolution of firms’ structural variables, such as asset and debt values, to determine the time of default. These last models are based on capital structure theory of Modigliani and Miller \cite{modigliani1958cost} and option pricing theory of Black and Scholes \cite{black1973pricing} and Merton \cite{merton1974pricing} and rely heavily on diffusion processes to model the evolution of the firm's asset value process \cite{sundaresan2013review}.

There have been some efforts in the literature to extend the Merton's structural model to more complex dynamical processes such as jump-diffusion \cite{fiorani2010single} or regime-switching diffusion processes \cite{siu2008pricing}. 
The family of Markov-modulated regime switching processes based on continuous-time Markov chains has been extensively used and analyzed by academics and market practitioners to model the variable economic conditions observed frequently in finance and economics (see e.g. \cite{Elliott,Hainaut} and the references therein). 

In the case that we are confronted with a sudden synchronous jump in the asset value process alongside the regime shifts, we obtain a process recently introduced into the field of asset pricing (see Chourdakis \cite{chourdakis2005switching} for a complete account) with some option pricing applications (see e.g. \cite{bastani2013radial,lee2014financial,rambeerich2016high}). It is shown empirically that this regime-switching model and its extensions could successfully be calibrated to a wide range of asymmetric volatility profiles (see e.g. \cite{assonken2017modeling}).

In a recent contribution, Hainaut and Colwell \cite{hainaut2016structural} studied a synchronous jump regime-switching version of Merton's structural credit risk model to estimate the probability of default and price corporate bonds. 
Based on the fact that there is no closed-form solution for default probabilities and bond prices in this setting, they developed a numerical approximation scheme based on the Fourier space time-stepping method (see e.g. \cite{jackson2007option,jackson2008fourier}) which is only applicable for L\'{e}vy processes with known characteristic functions. 

In this paper, we extend the family of L{\'e}vy processes utilized in Hainaut and Colwell \cite{hainaut2016structural} by considering the family of tempered stable distributions which capture the behaviour of firm's asset value process more accurately. Tempered stable distributions and their associated processes are a class of models that have attracted the attention of many researchers from applied probability and stochastic analysis to physics and financial mathematics. They were introduced in \cite{sato1999levy}, where the associated L\'evy processes are called ``truncated L\'evy flights'' and have been generalized by several authors \cite{kuchler2013tempered}. Tempered stable distributions form a six parameter family of infinitely divisible distributions, which cover several well-known sub-classes such as Variance Gamma (VG) distributions  \cite{carr2002fine,madan1990variance}, bilateral Gamma distributions \cite{kuchler2013tempered} and CGMY distributions \cite{sundaresan2013review}. Properties of tempered stable distributions have been investigated, e.g., in \cite{madan2001purely,lee2016comparison,revuz2013continuous,bianchi2010tempered}. For financial modelling, they have been applied, e.g. in  \cite{siu2008pricing,rachev2011financial,tankov2003financial,duffie2012credit}, see also the recent book \cite{kuchler2008bilateral}.

Based on the fact that the system of partial integro-differential equations (PIDEs) arising from these models contains singular integral terms arising from the structure of their probability density functions are difficult to treat numerically and these singularities adversely affect the accuracy and convergence rate of the standard numerical methods (see e.g. \cite{bastani2013radial, almendral2007accurate, lee2016comparison, hirsa2012computational, cont2005finite, brummelhuis2014radial, itkin2017pricing, damircheli2019solution}), we need to employ special de-singularization techniques to overcome these difficulties with minimum extra computational cost. Some research studies have been done to address this issue in the literature. Cont and his coworkers \cite{la2003integro, cont2005finite,  tankov2003financial} have approximated the above-mentioned PIDEs with a new singularly perturbed PIDE($\epsilon$) where small jumps near the origin are estimated by an appropriate Brownian motion process. Exploiting the Fourier transform technique is another alternative employed for PIDEs with and without singularities \cite{surkov2009option}. However, another approach which is inspired form physical intuition is to convert the original PIDEs into semi-parabolic PDEs by eliminating the integral terms and applying a modern finite difference method to solve the derived problem \cite{itkin2017pricing}. Last but not least, using a stable method to approximate the solution of the PIDEs such that  the numerical method inherently handles the singularity is similar to the method proposed in this paper.

In recent years, considerable attention has been centered around the development of new numerical schemes under the heading of mesh-free or meshless methods. They constitute powerful tools in the field of numerical analysis with a wide range of applications from image processing to machine learning \cite{fasshauer2007meshfree}. Among these meshless methods, those which are based on expansion in terms of radial basis functions (RBFs) provide a simple and efficient framework to implement the meshfree idea and offer promising results in real world applications. As a method to solve partial (integro-) differential equations arising from a wide variety of applications, RBF-based methods are used both in Galerkin and collocation modes with good accuracy and stability properties (see e.g.  \cite{kormann2013galerkin,sarra2008numerical} and references therein).  

In this paper, we propose and analyze an RBF-based collocation scheme based on global multiquadratic basis functions with inherent capability to desingularize the integral terms in a  straightforward manner. Our contribution could be considered as an extension of the method presented in Brummelhuis and Chan \cite{brummelhuis2014radial} to the more general setting of tempered stable L\'evy processes where in addition, we have provided a firm theoretical basis on the working of the proposed method. In order to validate the accuracy and efficiency of this method, we consider the default probability estimation for three different firms previously studied in \cite{hainaut2016structural} now based on a tempered stable L\'evy distribution as the firm's asset value process. Our findings confirm the capabilities of the new method as a reliable and accurate scheme to handle the problem complexities. 

The outline of the article is as follows. In Section 2, we introduce the basic ingredients of our proposed model for the asset value process including a brief account of the hidden Markov processes underlying the regime change behaviour of the market and encoding the information about changing economic conditions. We then provide some general comments about structural models to compute the default probability of a zero coupon bond. Finally, we describe the details of the underlying Markov-modulated switching L\'evy processes with synchronous jumps to model the asset value process of the ﬁrm under the risk-neutral measure. 
In Section 3, we review the structural model for corporate debts and the features of the firm’s asset value process modeled as an exponential switching Lévy process. Among the numerous Lévy processes Next, the system of partial integro-differential equations (PIDEs) driving the probabilities of default is built. Sections 4 present the numerical method based on the proposed radial basis function collocation. In Section 5, some numerical experiments are conducted to illustrate the efficiency and accuracy of the proposed method. We conclude the paper by pointing out some future research directions.








\section{Proposed Model for the Value of Firm's Assets}
In order to model the firm’s asset value process, we employ a Markov-modulated exponential L\'{e}vy process having synchronous jumps with regime shifts. Before presenting the details of this process, we first remind the reader of some general properties of continuous-time Markov processes. In this respect, we assume that the economic state at time $t$ is modeled by a continuous-time hidden Markov process, $\alpha_t$, with values from the set $\mathcal{H} = \{1,2,\cdots, H\}$, each representing an economic state or regime. 
Let $Q = (q_{ij})_{H\times H}$ denote the generator matrix of $\alpha_t$ where the off-diagonal entries are non-negative and $\sum_{j=1}^{H}q_{ij} = 0,$ for all $i\in \mathcal{H}$ (see e.g. \cite{Markovchain}). Discretizing the time variable with a sufficiently small step-size, $\Delta t$, we could show that $q_{i,j}\Delta t$ is the probability to switch from state $i$ to state $j$, for $i\neq j$. We also define the transition probability matrix of the process, denoted by $P(t,s)$ in the form
\begin{equation}\label{transition_matrix}
    P(t,s) = {\rm e}^{Q(s-t)},\quad s\geq t.
\end{equation}
We denote the elements of $ P(t,s)$ by $p_{ij}(t,s)$ and interpret them as being probabilities of switching from state $i$ at time $t$, to state $j$ at time $s$.
\subsection{Switching L\'{e}vy Processes with Synchronous Jumps}\label{subsection_Switching}
In the remainder, we consider an exponential L\'evy model for the asset value process of the form 
\begin{equation*}
    V_t = V_0e^{X_t},
\end{equation*}
in which $X_t$ is a Markov-modulated Lévy process\footnote{A L\'evy process is a stochastic process with stationary and independent increments which is continuous in probability (see e.g. Papapantoleon \cite{papantaleon2000introduction}).} with $X_0=0$. More precisely, $X_t$ is a piecewise L\'evy process depending upon the state, $\alpha_t$ of an economy with simultaneous jumps at regime switching times. Dynamics of the process $X_t$ could be written as 
\begin{equation}
    dX_t=dX^{\alpha_t}_t+J_{\alpha_t^{-},\alpha_t},
\end{equation}
where $X_t$ has jumps represented by $J_{\alpha_t^{-},\alpha_t}$ when regime switches at time $t$. The processes $X^j_t$, $j = 1,\cdots H$ are independent L\'evy processes with a L\'evy-It\^o decomposition of the form
\begin{equation}
\mathrm{d} X_{t}^{j}=\mu_{j} \mathrm{d} t+\sigma_{j} \mathrm{d} W_{t}^{j}+\int_{|z|>1} z J_{X^{j}}(\mathrm{d} t, \mathrm{d} z)+\int_{|z| \leq 1} z\left(J_{X^{j}}(\mathrm{d} t, \mathrm{d} z)-\nu(j, \mathrm{d} z) \mathrm{d} t\right),
\end{equation}
in which $W^j_{t}$ is a standard Brownian motion on the underlying probability space and $J_{X^j}(t,z)$ is a jump process of intensity $\nu(j,\cdot)$ (different from $J_{\alpha_t^{-},\alpha_t}$), the L\'{e}vy measure of $X_t$ in state $j$. The triplet $(\mu_j,\sigma_j,\nu(j,z))$ uniquely determines the characteristic function of $X^j_t$ by the L\'{e}vy-Khintchine theorem (see e.g. \cite{tankov2003financial, sato1999levy}) stating that 
\begin{equation}
    {\mathbb E}\left({\mathrm e}^{{\mathrm i} \omega X_{t}^j}\right)={\mathrm e}^{t \Psi_j(\omega)},
\end{equation}
in which the characteristic exponent of the L{\'e}vy process is given by
\begin{equation}\label{levy_exponent}
   \Psi_j(w)={\mathrm i} \mu_j \omega-\frac{\omega^{2} \sigma_j^2}{2}+\int_{ \mathbb{R}}\left({\mathrm e}^{{\omega} z}-1-{\mathrm i} \omega z \mathbbm{1}_{|z|<1}\right) \nu(j,d z). 
\end{equation}
In the above expressions, $\mu_{j}$ is the drift, $\sigma_{j}>0$ is the diffusion and $\nu(j,\cdot)$ is the L{\'e}vy measure on $\mathbb{R}$ satisfying  $\nu(j,\{0\})=0$ and 
$\int_{\mathbb{R}}(1 \wedge |z|^2) \nu(j,dz) < \infty$ (see e.g. \cite{sato1999levy} for more details on L\'evy processes). 


The finite-activity L\'evy density of synchronous jumps $J_{i,j}$ is given by (see \cite{hainaut2016structural})
\begin{equation}
 \mu (i,j,z) = \left\{ {\begin{array}{*{20}{c}}
{{\eta _{i,j}}{{\rm{e}}^{ - {\eta _{i,j}}\left| z \right|}},}&{z \in {\mathbb{R}^{{\rm{sign}}({J_{i,j}})}},}\\
{0,}&{{\rm{otherwise,}}}
\end{array}} \right.
\end{equation}
which can take positive, negative or real values with a characteristic function of the form
\begin{equation}
\begin{aligned}
{\Theta ^{ij}}(u) &=\mathbb{E}\left(\mathrm{e}^{\mathrm{i} u J_{i, j}}\right) 
=\int_{-\infty}^{+\infty} \mathrm{e}^{\mathrm{i} u z} \mu(i, j, z) \mathrm{d} z = \frac{{{\eta _{i,j}}}}{{{\eta _{i,j}} - {\rm{i}}u{\rm{sign}}({J_{i,j}})}}.
\end{aligned}
\end{equation}
In the next subsection, we introduce the specific L\'evy process used in this paper along with its L\'evy measure.
\subsection{Generalized Tempered Stable L\'evy Processes}
Consider a L{\'e}vy process, $X^j_t$, specified by it's characteristic triplet, $(\mu_j,\sigma_j,\nu(j,z))$. The process is called of finite activity when $\nu(j,\mathbb{R})=\int_{\mathbb{R}} \nu(j,dz) < \infty$ and of infinite activity otherwise (see e.g. \cite{tankov2003financial}). We know also that a L{\'e}vy process can be decomposed in the form $X_t^j = \mu_j t + \sigma_j W^j_t + J^j_X$ where $J^j_X$ describes a jump process that its jumps can be finite or infinite.

Let the family of functions, $\bf{W}$, to contain those $w: (0, \infty) \rightarrow [0, \infty)$ which are continuous and also satisfy the following conditions:
\begin{itemize}
    \item They are decreasing, 
    \item $w (0^+) = 1$, 
    \item $\mathop {\lim }\limits_{x \to \infty } {z^n}w(z) = 0$, \quad $\forall n \in \mathbb{N}$.
\end{itemize}        
        
\begin{defn}\label{definition_tempered_function}
A L\'evy process is of generalized tempered-stable (GTS) type if its L\'evy measure, $\nu(j,\cdot)$, could be expressed as
\begin{equation}\label{density_GTS}
 {\nu}(j,dz): = \left( {\frac{{{C^j_ + }w({\beta^j_ + }z)}}{{{z^{{1} + {\alpha^j _ + }}}}}{\mathbbm{1}_{z > 0}}(z) + \frac{{{C^j_ - }w({\beta^j _ - }\left| z \right|)}}{{{{\left| z \right|}^{1 + {\alpha^j _ - }}}}}{\mathbbm{1}_{z < 0}}(z)} \right) dz,   
\end{equation}
in which $w \in \bf{W}$ and the parameters $\alpha^j_{\pm} <2$ and $C^j_{\pm }, \beta^j_{\pm} > 0$ are chosen such that $|z|^n {C_{\pm}^j }w({\beta^j_{\pm} }z)$ is bounded on $\mathbb{R}$ for each $n \in {\mathbb{N}}$. 
\end{defn}
\begin{rem}\label{finite_density}
It could be shown that for $\alpha^j_{\pm} < 2$, we have
\begin{equation}\label{intb}
\int_{|z|>1} |z|^n \nu(j,dz) < \infty,\quad \forall \ n \geq 0. 
\end{equation}
Indeed, the relation (\ref{intb}) is a necessary and sufficient condition for moments of a GTS L{\'e}vy process to exist \cite{tankov2003financial}. 
\end{rem}
\begin{table}[ht]
\centering
\caption{Classification of L{\'e}vy processes based on their L{\'e}vy–Khintchin representation. In this table, we introduce the parameters $\tilde{\mu_j}=\mu_j-\int_{|z|<1}z\nu(j,d z)$ and $\hat{\mu}_j = \tilde{\mu}_j - \int_{\mathbb{R}} \nu(j,d z)$  (For details see \cite{sato1999levy, tankov2003financial, kuchler2013tempered, brummelhuis2014radial}).}.
\label{change-levy-khintchin}
\resizebox{\linewidth}{!}{
\begin{threeparttable}
     \begin{tabular}{c|c||c|c|c|c}
     \myrowcolour 
     \multicolumn{6}{c}{\textbf{Different types of L{\'e}vy jump-diffusion processes and their L{\'e}vy–Khintchin representation}} \\ 
\hline 
     \begin{tabular}[c]{@{}c@{}} Activity\\ of $X^j_t$ \end{tabular} & \begin{tabular}[c]{@{}c@{}} Variation\\ of $J^j_X$ \end{tabular}   & Characteristic exponent $\Psi_j(w)$ & \begin{tabular}[c]{@{}c@{}} Parameters\\ for GTS\tnotex{tnote:robots-r1} \end{tabular} & Financial Models & \begin{tabular}[c]{@{}c@{}} Financial \\ Refs. \end{tabular} \\ 
\hline \hline
      Finite  & Finite\tnotex{tnote:robots-r2}  & \cellcolor{white} ${\mathrm{i}} \hat{\mu}_j w +\dfrac{ \sigma_j^{2} w^{2}}{2}+\int_{\mathbb{R}}e^{{\mathrm{i}} \omega z} \nu(j,d z)$     &  \cellcolor{white} $\alpha^j_{\pm} < 0$ & \cellcolor{white} \begin{tabular}[c]{@{}c@{}} Merton\\ and \\ Kou\end{tabular}   & \cellcolor{white} \begin{tabular}[c]{@{}c@{}} \cite{merton1976option}\\ and \\ \cite{kou2002jump}\end{tabular} \\ \cline{1-6} 
    \multirow{-1}{*}{\rotatebox[origin=c]{90}{Infinite\ \ }} &   Finite    & \cellcolor{white}  ${\mathrm{i}} {\tilde{\mu}_j} w+\dfrac{{\sigma_j}^{2} w^{2}}{2}+\int_{\mathbb{R}}\left(e^{{\mathrm{i}} \omega z}-1\right) \nu(j,d z)$   &     \cellcolor{white}  $\alpha^j_{\pm} \in [0,1)$  & \cellcolor{white} \cellcolor{white}\begin{tabular}[c]{@{}c@{}}VG\\ and \\ CGMY\end{tabular} & \cellcolor{white} \begin{tabular}[c]{@{}c@{}} \cite{madan1990variance}\\ and \\ \cite{carr2002fine}\end{tabular} \\ \cline{2-6}
    &  Infinite  &   \cellcolor{white} ${\mathrm{i}} \mu_j w + \dfrac{\sigma^{2} w^{2}}{2}+\int_{\mathbb{R}}\left(e^{{\mathrm{i}} \omega z}-1-{\mathrm{i}}wz\mathbbm{1}_{|z|\leq1}\right) \nu(j,d z)$  &\cellcolor{white} $\alpha^j_{\pm} \in [1,2)$  & \cellcolor{white} \begin{tabular}[c]{@{}c@{}}\ \\ CGMY \\ \ \end{tabular}  & \cellcolor{white} \cite{carr2002fine} \\ 
\end{tabular}
    \begin{tablenotes}
      \item\label{tnote:robots-r1} Note that $M_{GTS}$ is not a L{\'e}vy density when $\alpha^j_+\geq 2$ or $\alpha^j_- \geq 2$.
      \item\label{tnote:robots-r2} If a L{\'e}vy jump-diffusion process is of finite activity, then it has also jumps of finite variation.
    \end{tablenotes}
\end{threeparttable}
}
\end{table}
\begin{rem}
The following particular cases are known in the literature:
\begin{itemize}
 \item $C^j_+ = C^j_-$ and $\alpha^j_+=\alpha^j_{-} = 0$ is a Variance Gamma (VG) distribution, see \cite{madan2001purely, madan1990variance};
      \item $C^j_+ = C^j_-$ and $\alpha^j_+=\alpha^j_{-}$ is a CGMY-distribution, see \cite{carr2002fine}, also called classical tempered stable distribution in \cite{rachev2011financial};
      \item $\alpha^j_+=\alpha^j_{-}$ is a KoBol distribution, see \cite{boyarchenko2000option};
      \item  $\alpha^j_+ = \alpha^j_-$ and $\beta^j_+=\beta^j_{-}$ is the infinitely divisible distribution associated to a truncated {L\'e}vy flight, see \cite{koponen1995analytic};
    \item $\alpha^j_+=\alpha^j_{-} = 0$ is a bilateral Gamma distribution, see \cite{kuchler2008bilateral}.
\end{itemize}   
\end{rem}
\subsection{The Structural Model of Default Risk}
Let us consider a firm whose capital structure consists of debt and equity where the debt component is issued as a single zero-coupon bond, $B$ with principal value $L$. We also assume that the firm can only default at the maturity of debt, $T$. Assume also a risk-neutral probability space $(\Omega,\mathcal{F},\mathbb{Q})$ equipped with some filtration, $\mathcal{F}_t$. On this filtration is
defined the firm’s value process, $V_t$, that represents the total value of the firm’s assets. The probability measure $\mathbb{Q}$ is taken to be a risk neutral measure but the following results are
also valid under the real historical measure, $\mathbb{P}$.

We define the default event to occur only if the total value of the firm’s assets at the maturity is not sufficient to cover the loan's redemption:
\begin{equation}
\textmd{Default Event} = \{V_T < L\},
\end{equation}
and the probability of default is denoted by
\begin{equation}
D(t, T) = \mathbb{P}(V_T \leq L |{\cal F}_t).
\end{equation}
In case of bankruptcy, a fraction $R(\alpha_T)$, called the recovery rate, of asset value is repaid to debt-holders. 

In this work, as in Guo et al. \cite{guo2009modeling}, the debt recovers at a different magnitude
depending here upon the economic regime, $\alpha(t)$. If $r$ is the constant risk-free rate of interest, then the time $t$ price of a defaultable bond issued by the firm will be obtained as 
\begin{equation}
    B(t, T) = {\mathbb{E}}({\rm e}^{-r(T-t)}(R(\alpha_T)V_T \mathbbm{1}_{V_T<L} + L\mathbbm{1}_{VT\geq L}) |{\mathcal{F}}_t).
\end{equation}
We will present the approach used to compute the probabilities of default and the defaultable bond prices in the reminder.

\section{Derivation of the PIDEs for Default Probabilities}

In order to derive the PIDEs corresponding to the default probability of a given firm, there exists two main approaches which are usually exploited and give equal results:
\begin{enumerate}
    \item Using It\^o’s formula for general semi-martingales (or infinitesimal generator) \cite{tankov2003financial};
    \item Using the L{\'e}vy–Khintchine formula and the tempered Fourier transform (see Proposition (1.9) of \cite{revuz2013continuous}).
\end{enumerate}
The second approach is more useful in characterizing some intrinsic properties of the integro-differential operator of the problem which will be needed in later parts of the paper.  

The probability of default of a firm in state $\alpha_t$ and time $t$ under the measure $\mathbb{P}$ is defined as 
\begin{equation}
\begin{aligned}
D(t, T,X_t, {\alpha_t}) &=\mathbb{P}\left(V_{T} \leq L | \mathcal{F}_{t}\right)
=\mathbb{E}\left(\mathbbm{1}_{V_{T} \leq L} | \mathcal{F}_{t}\right),
\end{aligned}
\end{equation}
in which $D\left(t, T,x,j\right)$ is the probability of default at regime $j$ when the state process $X_t$ takes the value $x$. It could be shown (see e.g. Hainaut and Colwell \cite{Hainaut}) that $D\left(t, T, x,j\right)$ for $j=1,2,\cdots,H$ satisfy the following system of partial integro-differential equations 
\begin{equation}\label{PIDE1}
\left\{\begin{array}{ll}
\dfrac{\partial}{\partial t} D(t,T, x,j)+\mathcal{L} D(t,T, x,j)=0, \hspace{2cm} j=1 \ldots H,\\
\\
D(t, T, x,j)=\mathbbm{1}_{x<\ln \left(L / V_{0}\right)}, \hspace{3.35cm} j=1 \ldots H,\\
\\
\lim_{x \rightarrow -\infty}D(t,T,x,j) = 1,\hspace{3.3cm} j=1 \ldots H,\\
\\
\lim_{x \rightarrow +\infty}D(t,T,x,j) = 0,\hspace{3.3cm} j=1 \ldots H,
\end{array}\right.
\end{equation}
in which  $\mathcal{L}$  is the infinitesimal generator of default probabilities of the form (see e.g. \cite{hainaut2016structural,tankov2003financial})
\begin{equation}
\begin{aligned}
\mathcal{L} D(t,T,x,j)=& \mu_{j} \frac{\partial}{\partial x} D(t,T,x,j)+\frac{\sigma_{j}^{2}}{2} \frac{\partial}{\partial x^{2}} D(t,T,x,j) \\
&+\sum_{k \neq j} q_{j, k} \int_{\mathbb{R} \backslash\{0\}} D(t,T, x+z,k)-D(t,T,x,j) \mu(j, k, z) \mathrm{d} z \\
&+\int_{\mathbb{R} \backslash\{0\}} \left(D(t,T, x+z,j)-D(t,T,x,j)-z \mathbbm{1}_{|z| \leq 1} \frac{\partial}{\partial x} D(t,T,x,j)\right) v(j, \mathrm{d} z).
\end{aligned}
\end{equation}
Due to the relation $\sum_{i=1}^{H}q_{j,i} = 0$, the generator can be written as 
\begin{equation}\label{operator1}
\begin{aligned}
\mathcal{L} D(t,T,x,j)=& \mu_{j} \frac{\partial}{\partial x} D(t,T,x,j)+\frac{\sigma_{j}^{2}}{2} \frac{\partial}{\partial x^{2}} D(t,T,x,j) \\
&+q_{j,j}D(t,T,x,j)+\sum_{k \neq j} q_{j, k} \int_{\mathbb{R} \backslash\{0\}} D(t, T,x+z,k) \mu(j, k, z) \mathrm{d} z \\
&+\int_{\mathbb{R} \backslash\{0\}} \left(D(t,T, x+z,j)-D(t,T,x,j)-z \mathbbm{1}_{|z| \leq 1} \frac{\partial}{\partial x} D(t,T,x,j)\right) v(j, \mathrm{d} z).
\end{aligned}
\end{equation}

In order to make the presentation and implementation of the numerical meshfree scheme in the next section simple, we transform the PIDEs (\ref{PIDE1}) forward in time. To this end and using a common change of variables of the form $t = T - \tau$
and thus $u(\tau,x,j) = D(t,T,x,j)$ where $T$ is removed from $u$ for simplicity of notation, the problem (\ref{PIDE1}) will be transformed into
\begin{equation}\label{foreward-PIDE1}
\left\{\begin{array}{ll}
\dfrac{\partial}{\partial \tau} u(\tau,x,j)=\mathcal{L} u(\tau,x,j), \hspace{2.2cm} j=1 \ldots H,\\
\\
u(0, x,j)=\mathbbm{1}_{x<\ln \left(L / V_{0}\right)}, \hspace{2.6cm} j=1 \ldots H,\\
\\
\lim_{x \rightarrow -\infty}u(\tau,x,j) = 1,\hspace{2.5cm} j=1 \ldots H,\\
\\
\lim_{x \rightarrow +\infty}u(\tau,x,j) = 0,\hspace{2.5cm} j=1 \ldots H.
\end{array}\right.
\end{equation}

Developing a stable numerical scheme to solve Eq. (\ref{PIDE1}) is a necessity as the analytical solution of the PIDE system is not available in closed form.
In the sequel, we will propose our approach to approximate the solution of these systems of partial integro-differential equations. 


\section{De-Singularized Meshfree Approximation}\label{de-singular}
The important point to notice in confronting with the infinite activity processes is the singularity of the L{\'e}vy measure $\nu(j,dz)$ near the origin due to small jumps. This L\'evy measure for different processes behaves differently near zero. For instance, the L\'evy measure for the CGMY process (or equally the GTS process with parameters $C^j_{+}=C^j_-$ and $\alpha^j_{\pm} > 0$) approaches infinity much faster than the variance gamma (VG) process \cite{carr2002fine}. For this reason we should have the highest possible order of approximation for the integrand near zero.

Indeed, as we move in Table \ref{change-levy-khintchin} from top rows to bottom, more complex L\'evy processes appear due to the limitations which appear in separating the three integrands (see e.g. the third column in this Table). So we will need more complex numerical algorithms which will treat the singular behaviour more efficiently (see e.g. \cite{bastani2013radial, almendral2007accurate, lee2016comparison, hirsa2012computational, cont2005finite, brummelhuis2014radial, itkin2017pricing}).  

Based on the above discussion, various numerical algorithms (mainly based on finite differences) are proposed based on different kinds of integral terms appearing in PIDEs. Some researchers in the field have tried to offer comprehensive algorithms to tackle different L\'evy densities both with or without singularity. Following is the list of outstanding works in this subject:
\begin{itemize}
    \item Approximating small jumps near the origin using an appropriate approximated Brownian motion which will result in a new PIDE. Since this PIDE is a function of $\epsilon$, PIDE($\epsilon$) is used to refer to this derived PIDE. When $\epsilon$ goes to $0$, it is proved that the solution of PIDE($\epsilon$) will converge to the solution of original PIDE (\cite{la2003integro, cont2005finite,  tankov2003financial}). Applying this approach to the system (\ref{PIDE1}) will lead us to the following PIDE($\epsilon$) 
    \begin{multline}\frac{\partial}{\partial \tau} u+\left(\mu_{j}-\omega_{j}(\varepsilon)\right) \frac{\partial}{\partial x} u+\left(\frac{1}{2}\left(\sigma_{j}^{2}+\sigma_{j}(\varepsilon)\right) \frac{\partial^2 u}{\partial x^{2}}+q_{j, j} u\right) \\
    +\sum_{k \neq j} q_{j, k} \int_{{\mathbb{R} /\{0\}}} u(x+z) \mu(j, k, d z)
    +\int_{|z|>\varepsilon}(u(\tau,j,x+z)-u(\tau,j,x)) \nu(j, d z)=0,
\end{multline}
where the factors $\omega_{j}(\varepsilon)$ and $\sigma_{j}(\varepsilon)$ have been introduced in  \cite{tankov2003financial} (Chapter 12 ).
\item An alternative methodology to handle such  singularities is the famous Fourier Transform technique in which the PIDE (with or without singularity) is transformed into a system of ODEs for which efficient numerical approximation algorithms exist. This method is usually used to benchmark the financial problems (for more details, we refer the reader to \cite{surkov2009option}).
   \item Last but not least, we will mention the method used by Itkin et al. \cite{itkin2017pricing} in which by eliminating the integral part form original PIDEs, we deal with a semi-parabolic PDE. where modern finite difference methods are used to approximate the solution of this semi-parabolic PDE. It is worth noticing that these family of PDEs have a complex structure so that the standard numerical algorithms are not able to approximate their solutions. For details refer to 
\end{itemize}
Our proposed approach is to use a meshless collocation method based on RBFs of infinite or finite smoothness. Similar to Brummelhuis and Chan (see Section 3.3 of \cite{brummelhuis2014radial}), we show that smoothness properties  of radial basis functions such as MQ can overcome the singularity of a wide range of important L\'evy densities.
\subsection{Problem Discretization}
In this section, we describe a general numerical scheme based on RBF collocation to numerically solve the transformed linear partial integro-differential equation formulation (\ref{PIDE1}) 
which could be represented as
\begin{equation}
U_{\tau}={\cal L}U,
\end{equation}
in which 
\begin{equation}
U(\tau,x):=[u(\tau,x,1), u(\tau,x,2), \cdots, u(\tau,x,H)]^T.
\end{equation}
Many numerical schemes for time-dependent PDEs separate the discretization of time and space variables into distinct phases and develop the theory by assuming one discretization (outer discretization) to be carried out first leading to a so-called {\it semi-discrete problem}. After investigating the thus arising type of problem, one continues to perform the second discretization (inner discretization), ending up with a fully discrete scheme (see e.g. \cite{bornemann1989adaptive}).

We follow this tradition by first discretizing the equation (\ref{foreward-PIDE1}) in time and then using Newton's method for linearizing the time-discrete problem at the PDE level. Such a technique transforms the nonlinear stationary PDE at each time level into a sequence of linear PDEs which could now be solved using RBF collocation based on multiquadric (MQ) radial basis functions (see \cite{fasshauer2002newton} and \cite{bastani2018multilevel} for more details).
\subsection{Semi-Discretization in Time}
\label{sec:4.1}
Given an equally spaced time grid $0=\tau_0<\tau_1<...<\tau_{N} = T$ which subdivides the interval $[0,T]$ into $N$ sub-intervals of the form $[\tau_n, \tau_{n+1}],~n=0,1,2,\cdots,N-1$ with $\tau_n=n\Delta \tau$ and $\Delta \tau = T/N$, we first employ the temporal semi-discretization
$U(\tau_n,x)\approx U_n(x)$

\begin{equation}\label{NE}
\frac{{U_{n+1}(x) - U_n(x)}}{{\Delta \tau}} = \theta {\cal L}[{U}]\big|_{\tau=\tau_{n}}+(1 - \theta ){\cal L}[{U}]\big|_{\tau=\tau_{n+1}},
\end{equation}
to arrive at
\begin{equation}\label{rearranged}
[1 - \Delta \tau(1 - \theta ){\cal L}]{U}_{n+1}(x) =  [1 + \Delta \tau \theta {\cal L}]{U}_n(x),\quad n = 0,1,\cdots,N-1,
\end{equation}
in which $\theta \in [0,1]$ is the {\it implicitness level} of the scheme.

This last expression could be re-phrased as a linear elliptic partial differential equation of the form
\begin{eqnarray}\label{disct}
{\cal L}_{\Delta \tau}U_{n+1}(x) = {F}_{n},\quad n=0,\cdots,N-1,\label{T-PDE1}
\end{eqnarray}
with
\begin{equation}\label{EllipticAndRHS}
    {\cal L}_{\Delta \tau}:=1 - \Delta \tau(1 - \theta ){\cal L}, \quad
    { F}_{n}:=[1 + \Delta \tau\theta {\cal L}]U_{n}(x).
\end{equation}
\subsection{Meshfree RBF Collocation}
In order to give a rough idea of the radial basis function collocation, we first need the following concepts:
\begin{defn}\label{Radial}
A multivariate function $\Phi:{\mathbb R}^d\rightarrow {\mathbb R}$ is called \textsf{radially symmetric} or \textsf{isotropic} on ${\mathbb R}^d$, if $\Phi({\bf x})=\Phi({\bf y})$ whenever $\|{\bf x}\|=\|{\bf y}\|$. A radially symmetric function $\Phi(\cdot)$ could be represented as $\Phi(\bf{x})= \phi(\|\bf{x}\|)$, $\bf{x} \in$ ${\mathbb R}^d$ for some univariate function, $\phi: [0,\infty) \rightarrow \mathbb{R}$. Note that the norms will be the usual Euclidean 2-norm, defined by $\|{\bf x}\|_2:=\sqrt{{\bf x}^T{\bf x}}$.
\end{defn}

Before presenting the method, we first truncate the infinite spatial domain of the problem into a finite sub-domain of the form $I=[x_{min},x_{max}]$ and then consider the projection of $U(\tau,x)$ onto a finite dimensional space as a linear combination of radial basis function of the form
\begin{equation}\label{RBF2}
U_{n}(x)=\sum_{j=1}^{N_{x}} \Upsilon_{j}^{(n)} \phi\left(\left\|x-x_{j}\right\|\right) \equiv \sum_{j=1}^{N_{x}} \Upsilon_{j}^{(n)} \Phi_{j}(x),
\end{equation}
in which, $x_j$ for $j=1,\cdots, N_s$ is a set of scattered data points in the sub domain $I$ and $\Upsilon_{j}^{(n)}:=[\upsilon_{j1}^{(n)},\upsilon_{j2}^{(n)},\cdots,\upsilon_{jH}^{(n)}]^T$. Now, if we use $\theta = 0$ and Substitute (\ref{RBF2}) into (\ref{disct}), we obtain
\begin{equation}
\sum_{j=1}^{N_{x}} \Upsilon_{j}^{(n+1)} \Phi_{j}(x)-\Delta \tau \sum_{j=1}^{N_{x}}  \Upsilon_{j}^{(n+1)}\mathcal{L} \Phi_{j}(x)=U_{n}(x).
\end{equation}
So we have
\begin{equation}
    (\Phi-\Delta \tau\Phi_\mathcal{L})\Upsilon^{(n+1)}={\bf F}^{(n)}.
\end{equation}
where 
%
\begin{align}
     &\Phi=\left[\Phi_{j}(x_i)\right]_{i,j=1,\cdots,N_x},&\\
     &\Phi_\mathcal{L}=\left[\mathcal{L} \Phi_j(x_i)\right]_{i,j=1,\cdots,N_x},&\\
          &{\bf F}^{(n)} := [U_{n}(x_1),U_{n}(x_2),\cdots,U_{n}(x_{N_x})],&\\
     & \Upsilon^{(n+1)} := \left({\Upsilon_1^{(n+1)}}^T,{\Upsilon_2^{(n+1)}}^T, \ldots,{\Upsilon_{N_x}^{(n+1)}}^T \right)^T.&
 \end{align}
It is worth mentioning that mega matrix $\Phi_\mathcal{L}$ has a block structural where off-diagonal blocks entangles data from different regimes (see for more details \cite{bastani2013radial}).

\subsection{Some Basic Facts about Radial Basis Function Interpolation}\label{sec:4.2}
Let us recall some fundamental properties of RBFs before presenting the core theorem of this section.

\subsection{De-Singularization Procedure}
The first part of this subsection is to review some of the necessary concepts and theorems which we will apply to prove the main theorem of this article.
\begin{defn}\label{PDef} (Bochner \cite{bochner1955harmonic})
A radially symmetric multivariate function $\Phi:{\mathbb R}^d\rightarrow {\mathbb R}$ is called \textsf{positive definite}, if for any finite $k\in{\mathbb N}$ pairwise different points ${\bf x}_1,{\bf x}_2,\cdots,{\bf x}_k$ in ${\mathbb R}^d$ and for each ${\bf c}=[c_1,c_2,\cdots,c_k]^T\in{\mathbb R}^k$, we have
\begin{equation}\label{quadratic}
\sum_{i=1}^k\sum_{j=1}^kc_ic_j\Phi({\bf x}_i-{\bf x}_j)\geq 0.
\end{equation}
The function $\Phi$ is called \textsf{strictly positive definite} on ${\mathbb R}^d$, if the quadratic form in the left hand side of (\ref{quadratic}) is zero only for ${\bf c}={\bf 0}$. 
We also call a univariate function $\phi: [0,\infty) \rightarrow \mathbb{R}$ strictly positive definite on $\mathbb{R}^d$ (abbreviated as $\phi\in {\bf SPD}_d$), if the corresponding radially symmetric multivariate function defined by $\Phi(\bf{x}):= \phi(\|\bf{x}\|)$, $\bf{x} \in$ ${\mathbb R}^d$ is strictly positive definite.
\end{defn}
In order to prove our main result, we need the following properties of ${\bf SPD}_d$ functions which play a key role in our development. 
\begin{thm}\label{properties-SPD} (Fasshauer \cite{fasshauer2007meshfree})
Every ${\bf SPD}_d$ function, $\Phi(\cdot)$, satisfies the following properties:  
\begin{itemize}
    \item $\Phi(0) \geq 0$;
    \item $|\Phi(x)| \leq \Phi(0)$;
    \item $\Phi$ is a positive function.
\end{itemize}
\end{thm}
\begin{thm}\label{derivatives-SPD} (Buescu and Paix\!{a}o \cite{buescu2011differentiability} and Massa and et al. \cite{massa2017estimates})
Let $\Phi: \mathbb{R}^m \rightarrow \mathbb{R}$ be a positive definite function and suppose $\Phi$ is of class $C^{2n}$ (resp. $C^{\infty}$) in some neighborhood of the origin for some positive integer $n$. Then
\begin{itemize}
    \item $\Phi \in C^{2n}(\mathbb{R}^m)$ (resp. $C^{\infty}(\mathbb{R}^m))$,
    \item ${\left| {{{\rm{D}}^{\alpha  + \beta }}\Phi (x)} \right|^2} \le {( - 1)^{{m_1} + {m_2}}}{{\rm{D}}^{2\alpha }}\Phi (0){{\rm{D}}^{2\beta }}\Phi (0),\quad \forall x \in \mathbb{R}^m,\quad |\alpha|, |\beta| \leq n$,
    \item In the special case where $m=1$, we have $\Phi \in C^{2n}(\mathbb{R})$ and for all integers $m_1$, $m_2$ with $0 \leq m_i \leq n$, $i = 1, 2$ and every $x \in \mathbb{R}$ we have
\[
{\left| {{\Phi ^{({m_1} + {m_2})}}(x)} \right|^2} \le {( - 1)^{{m_1} + {m_2}}}{\Phi ^{(2{m_1})}}(0){\Phi ^{(2{m_2})}}(0).
\]
\end{itemize}
\end{thm}

 
\begin{rem}\label{points}
Based on the above results, we have the following:
\begin{itemize}
    \item[{\bf (i)}] Every ${\bf SPD}_d$ function, $\Phi$, is bounded above (see the Theorem \ref{properties-SPD}).
    \item[{\bf (ii)}] Based on above assumptions, $\Phi \in {\bf SPD}_d \cap C^{2n}(\mathbb{R})$ for $n\geq 2$ ($\Phi$ is $2n$-times differentiability on $\mathbb{R}$).
    \item[{\bf (iii)}] Based on the Theorem \ref{derivatives-SPD}, if we define ${M_k}: = \mathop {\max}\limits_{x \in \mathbb{R}} \{ | {{\Phi ^{(k)}}(x)} |\}$ for $0 \leq k \leq 2n$, then $M_k < \infty$.
\end{itemize}
\end{rem}

\begin{thm}\label{main_result1}
Let $\Phi : \mathbb{R} \rightarrow \mathbb{R}$ be a real-valued strictly positive definite function and assume that $\Phi$ is of class $C^{2n}$ in some neighborhood of the origin for some positive integer $n \geq 2$. Then, we have $\Phi \in  C^{2n}(\mathbb{R})$. Also for any tempered stable L\'{e}vy measure $\nu(i,\mathrm{d} z)$ with $\alpha^j_{\pm} < 2$ defined in \ref{density_GTS}, the integral 
\begin{equation}
I_v(x) := \int_{\mathbb{R} } \left( \Phi(x+z)-\Phi(x)-z \mathbbm{1}_{|z| \leq 1} \frac{\partial}{\partial x} \Phi(x)\right) \nu(j, \mathrm{d} z),
\end{equation}
is finite for all $x \in \mathbb{R}$ and without singularity at the origin. 
\end{thm}
\begin{proof}
For arbitrary constants $0 < \varepsilon < 1$ and $b \gg 1$, the improper integral $I_v(\cdot)$ could be written as
\[
I_v =  \mathop {\lim }\limits_{b \to \infty }\left( I^{(1)}_v + I^{(4)}_v \right) + \mathop {\lim }\limits_{\varepsilon \to 0^+ }\left( I^{(2)}_v + I^{(3)}_v \right),
\]
where 
\begin{align*}
  I^{(1)}_v &= \int_{-b}^{-1} \left( \Phi(x+z)-\Phi(x)\right) \frac{{{C^j_ - }w({\beta^j _ - }\left| z \right|)}}{{{{\left| z \right|}^{1 + {\alpha^j _ - }}}}} \mathrm{d} z,&\\
  I^{(2)}_v&= \int_{-1}^{-\varepsilon} \ \left( \Phi(x+z)-\Phi(x)-z \mathbbm{1}_{|z| \leq 1} \frac{\partial}{\partial x} \Phi(x)\right) \frac{{{C^j_ - }w({\beta^j _ - }\left| z \right|)}}{{{{\left| z \right|}^{1 + {\alpha^j _ - }}}}} \mathrm{d} z,&\\
  I^{(3)}_v & = \int_{\varepsilon}^{1} \ \left( \Phi(x+z)-\Phi(x)-z \mathbbm{1}_{|z| \leq 1} \frac{\partial}{\partial x} \Phi(x)\right) \frac{{{C^j_ + }w({\beta^j _ + } z )}}{{{z^{{1} + {\alpha^j _ + }}}}} \mathrm{d} z,&\\  
  I^{(4)}_v& = \int_{1}^{b} \left( \Phi(x+z)-\Phi(x)\right) \frac{{{C^j_ + }w({\beta^j _ + }z)}}{{{z^{{1} + {\alpha^j _ + }}}}} \mathrm{d} z.& 
\end{align*}
According to the comparison test for improper integrals (see Chapter 12 in \cite{wrede2002schaum}), it is easy to show that the improper integrals of the first kind (i.e. $I_v^{(1)}$ and $I_v^{(4)}$) are absolutely convergent as $b$ tends to infinity for arbitrary $\alpha^j_{\pm} < 2$, because we can obtain the following upper bounds for the integrals $I^{(1)}_v$ and $I^{(4)}_v$ using Theorem \ref{properties-SPD} and Definition \ref{definition_tempered_function}:
\begin{align*}
 \left| I^{(1)}_v \right|& \leq  \int_{-b}^{-1} \left|\left( \Phi(x+z)-\Phi(x)\right) \frac{{{C^j_ - }w(-{\beta^j _ - } z)}}{{{{ (-z) }^{1 + {\alpha^j _ - }}}}}\right| \mathrm{d} z&\\
 &\leq  2 \Phi(0^+) \int_{-b}^{-1} \frac{C^j_- w(-{\beta^j _ - } z)}{{{{(- z) }^{1 + {\alpha^j _ - }}}}} \mathrm{d} z =  \Phi(0^+) \int_{b>|z|\geq 1} \nu_{\rm sym}(j,\mathrm{d} z),&
\end{align*}
in which $\nu_{\rm sym}(j,\mathrm{d} z)$ is the ``symmetrized'' version of the original L\'evy measure. Moreover, we have 
\begin{align*}
 \left| I^{(4)}_v \right|& \leq  \int_{1}^{b} \left|\left( \Phi(x+z)-\Phi(x)\right) \frac{{{C^j_ + }w({\beta^j _ + } z)}}{{{{ z }^{1 + {\alpha^j _ + }}}}} \right|\mathrm{d} z \leq  \Phi(0^+) \int_{1 \leq |z| < b} \nu_{\rm sym}(j,\mathrm{d} z).&
\end{align*}
According to Remark \ref{finite_density}, the right hand sides of the above two inequalities are absolutely convergent as $b \rightarrow \infty$. Note also that the symmetric L\'evy measure $\nu_{\rm sym}(j,\mathrm{d} z)$ is given by substituting  $C^j_+ = C^j_-$, $\alpha^j_+ = \alpha^j_-$ and $\beta^j_+=\beta^j_{-}$ in (\ref{density_GTS}). The proof will be complete if we show that the integrals $I^{(2)}_v$ and $I^{(3)}_v$ are also absolutely convergent. In this respect, we will simultaneously prove them for the case $\alpha_{\pm} < 2$, which consists of the following three different subcases: 
\begin{itemize}
    \item[(1)] Finite activity tempered stable L\'evy measures ($\alpha_{\pm} < 0$),
    \item[(2)] Infinite activity but finite variation tempered stable L\'evy measures ($0 \leq \alpha_{\pm} < 1$),
    \item[(3)] Infinite activity and infinite variation tempered stable L\'evy measures ($1 \leq \alpha_{\pm} < 2$).
\end{itemize}
Based on Definition \ref{definition_tempered_function} and also by applying Theorem \ref{properties-SPD}, Theorem \ref{derivatives-SPD}, Lemma \ref{points} and also Taylor's theorem, the following inequalities are obtained for the integrals $I^{(2)}_v$ and $I^{(3)}_v$:
\begin{align*}
 \left| I^{(2)}_v \right|& \leq  \int^{-\varepsilon}_{-1} \left|\left( \Phi(x+z)-\Phi(x)-z \frac{\partial}{\partial x}\Phi\right) \frac{{{C^j_ - }w({-\beta^j _ - } z )}}{{{{ (-z) }^{1 + {\alpha^j _ - }}}}} \right|\mathrm{d} z &\\
  &=  \int^{-\varepsilon}_{-1}\left| \left( \frac{z^2}{2!}\frac{\partial^2 }{\partial x^2}\Phi(x) + R_2(z)\right) \frac{{{C^j_ - }w({-\beta^j _ - } z )}}{{{{ (-z) }^{1 + {\alpha^j _ - }}}}} \right| \mathrm{d} z&\\
  &\leq   \frac{M_2}{2!} \int^{-\varepsilon}_{-1}  z^2  \frac{{{C^j_ - }w({-\beta^j _ - } z )}}{{{{ (-z) }^{1 + {\alpha^j _ - }}}}} \mathrm{d} z +  \frac{M_3}{3!} \int^{-\varepsilon}_{-1} |z|^3 \frac{{{C^j_ - }w({-\beta^j _ - } z )}}{{{{ (-z) }^{1 + {\alpha^j _ - }}}}}  \mathrm{d} z&\\
    &\leq   \frac{M_2}{4} \int_{\varepsilon \leq |z| < 1}  z^2   \nu_{\rm sym}(j,\mathrm{d} z) +  \frac{M_3}{12} \int_{\varepsilon \leq |z| < 1} |z|^3 \nu_{\rm sym}(j,\mathrm{d} z)&\\
    &\leq   \frac{1}{4} \max\{M_2,M_3\} \int_{\varepsilon \leq |z| < 1}  z^2   \nu_{\rm sym}(j,\mathrm{d} z), &   
\end{align*}
and
\begin{align*}
 \left| I^{(3)}_v \right|& \leq  \int^{1}_{\varepsilon} \left|\left( \Phi(x+z)-\Phi(x)-z \frac{\partial}{\partial x}\Phi\right) \frac{{{C^j_ + }w({\beta^j _ + } z )}}{{{{ z }^{1 + {\alpha^j _ + }}}}} \right|\mathrm{d} z &\\
  &=  \int^{1}_{\varepsilon}\left| \left( \frac{z^2}{2!}\frac{\partial^2 }{\partial x^2}\Phi(x) + R_2(z)\right) \frac{{{C^j_ + }w({\beta^j _ + } z )}}{{{{ z }^{1 + {\alpha^j _ + }}}}} \right|\mathrm{d} z&\\
    &\leq   \frac{M_2}{2!} \int_{\varepsilon}^{1}  z^2  \frac{{{C^j_ + }w({\beta^j _ + } z )}}{{{{ z }^{1 + {\alpha^j _ + }}}}} \mathrm{d} z +  \frac{M_3}{3!} \int_{\varepsilon}^{1} |z|^3 \frac{{{C^j_ + }w({\beta^j _ + } z )}}{{{{ z }^{1 + {\alpha^j _ + }}}}}  \mathrm{d} z&\\
    &\leq   \frac{M_2}{4} \int_{\varepsilon \leq |z| < 1}  z^2   \nu_{\rm sym}(j,\mathrm{d} z) +  \frac{M_3}{12} \int_{\varepsilon \leq |z| < 1} |z|^3 \nu_{\rm sym}(j,\mathrm{d} z)&\\
    &\leq   \frac{1}{4} \max\{M_2,M_3\} \int_{\varepsilon \leq |z| < 1}  z^2   \nu_{\rm sym}(j,\mathrm{d} z),& 
\end{align*}
in which $R_k(z)$ is the mean-value form of the remainder term in the Taylor's theorem and given by:
\begin{equation}\label{taylor-reminder}
  R_k(z)=\frac{f^{(k+1)}(\xi_x)}{(k+1)!}z^{k+1},  
\end{equation}
for some real number $\xi_x$ between $x$ and $x+z$. 

By the comparison test for improper integrals and according to the definition of a L\'evy measure (see Subsection \ref{subsection_Switching}), the integrals $I_v^{(2)}$ and $I_v^{(3)}$ are absolutely convergent due to the existence of upper integral bounds  as $\varepsilon \rightarrow 0^+$. Note that 
$|R_k(z)| \leq \frac{||f^{(k+1)}(\cdot)||}{(k+1)!}|z|^{k+1}$ and so the proof is complete.
\end{proof}

Theorem \ref{main_result1} shows that collocation based on infinitely differentiable radial basis functions such as Gaussian bases is a reliable method to approximate the solution of PIDEs arising from models including tempered stable L\'{e}vy  processes. I the next section, we assess this in practice for different standard test problems.  
\section{Numerical Experiments}

In this section, we use numerical experiments to showcase the use of the proposed meshfree methodology to find the default probability of firms whose asset price model is a Markov-modulated L\'{e}vy process with synchronous jumps.
Uniform distribution of points in interval $[x_{\rm min}, x_{\rm max}] = [-8,8]$ is used for the collocation procedure \cite{bastani2013radial}. Although, we analyze the methodology with different families of RBFs, we choose to report the results based on Gaussian RBFs due to the fact that the solution of a PIDEs based on the jump-diffusion model inherently behaves like a Gaussian distribution (see \cite{kazemi2018new} for more details). The first test problem is pertained to three French firms calibrated and fitted by Markov modulated L\'evy processes with synchronous jumps in \cite{Hainaut} for a period of 10 years from the first of January 2004 to the twenty forth of January 2014, whereas the two last test problems are designed by authors for CGMY L\'evy process and Kobol l\'evy process that are analyzed with  the proposed method in this paper.\\
\\
{\bf Test Problem I (VG Model)}
In the first experiment, our computations are based on parameters obtained in \cite{Hainaut} for three French companies: Axa, STMicroelectonics and soci\'et\'e G\'en\'erale  with the fitted parameters reported in Table \ref{dataExI}. Using econometric adjustment in \cite{Hainaut}, a regime switching process with two regimes ($H=2$) based on the VG process with synchronous jumps is used to model the dynamics of the firms's asset value process.

\begin{table}[ht]
\centering
\caption{Parameters of the VG model for  Axa, STMicroelectonics and Soci\'{e}t\'{e} G\'{e}n\'{e}rale firms.}
\label{dataExI}
     \begin{tabular}{ccccccccc}
\topline 
         %
     VG Parameters &  \multicolumn{2}{c}{Axa} & & \multicolumn{2}{c}{STM} & & \multicolumn{2}{c}{Soc Gen}  \\
\midtopline
                   & State 1 & State 2 & &State 1 & State 2 & & State 1 & State 2\\ 
\cline{2-3}   \cline{5-6}  \cline{8-9} 
            $\sigma_j$ & 0.4460 & 0.1234 & &0.2495 & 0.1534 & &0.3227 & 0.1675 \\
            $\theta_j$ & -0.1421 & 0.0196 & &-0.1135 & 0.0043 && -0.1576 & 0.0254\\
            $\kappa_j$ & 0.0236 & 0.0011 & & 0.0374 & 0.0015 & & 0.0306 & 0.0028\\
            $p_j$ & 0.8755 & 0.9664 & & 0.6822 & 0.9386 & & 0.8083 & 0.9549\\
            $\eta_{i,j}$ & 0.0160 & -0.0092 & & 0.0078 & -0.0074 & & 0.0132 & -0.0117\\
\bottomline
\end{tabular}
\end{table}


For this problem, the probability of default surfaces for the STMicrolectronics firm are illustrated in two different regimes in Figure \ref{surfSTEM}. Besides, we compare the results of meshfree collocation method of this paper with the result of FFT method proposed by Hainaut et al. in \cite{Hainaut} in Table \ref{errorRBFthreeFirms}. It can be seen that we obtain better results as the number of spatial nodes increases. Nevertheless the computational time will naturally increase as the size of the system becomes larger. The probability of bankruptcy for the three companies is depicted in Figure \ref{PDthreeFirms} for the time horizon from 1 to 10 years, where firms can default at the end of each financial year.

\begin{figure}[!ht]
\centering

 \includegraphics[width=10cm]{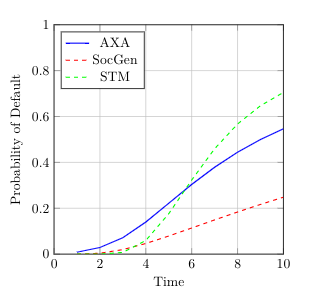}
\caption{}
\caption{Default Probability, VG models}
\label{PDthreeFirms}
\end{figure}


\begin{table}[ht]
        \centering
        \caption{The error incurred in estimating the probability of defaulat for Soci\'{e}t\'{e} G\'{e}n\'{e}rale.}
        \label{errorRBFthreeFirms}
      \begin{tabular}{c|c|c|c} \hline
           $N_s$ & Relative Error in State 1 & Relative Error in State 2&  CPU-Time (s)\\
           \hline
              64    &  0.638564   & 0.67354  & 0.814\\
             128    &  0.012792   &0.01321  & 0.955\\
              256   &  0.053485   & 0.05090  & 1.812\\
              512   &  0.001012   & 0.00103  & 2.336\\
         \hline
        
        \end{tabular}
\end{table}

\begin{figure}[ht]

\begin{subfigure}{0.5\textwidth}

 \includegraphics[width=7.5cm]{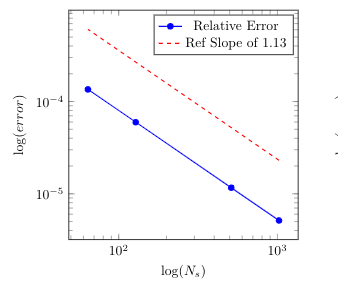}
\caption{}
\end{subfigure}
\begin{subfigure}{0.5\textwidth}

\includegraphics[width=7.5cm]{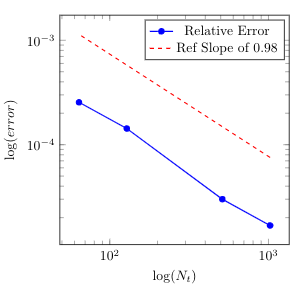}
\caption{}
\end{subfigure}

\caption{Accuracy properties of numerical method for Axa with VG models, a) Impact of space steps b) Impact of time steps}
\label{orderAxa}
\end{figure}


\begin{figure}[ht]

\begin{subfigure}{0.5\textwidth}

\includegraphics[width=7.5cm]{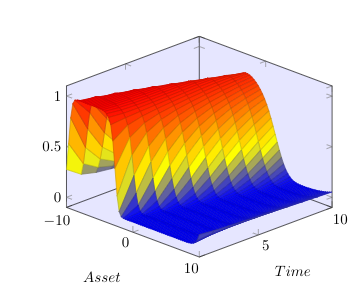}

\caption{}
\label{fig_PD_STEM_S1}
\end{subfigure}
\begin{subfigure}{0.5\textwidth}

\includegraphics[width=7.5cm]{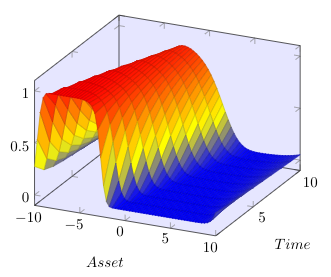}

\caption{}
\label{fig_PD_STEM_S2}
\end{subfigure}

\caption{Probability of default for STEM: (a) the first regime (b) the second regime.}
\label{surfSTEM}
\end{figure}

\begin{figure}[!ht]
\centering

%

\includegraphics[width=7.5cm]{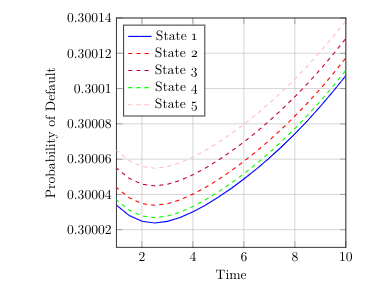}

\caption{Default Probability for the CGMY model.}
\label{PDCGMY}
\end{figure}



\begin{figure}[ht]
\centering

\includegraphics[width=10cm]{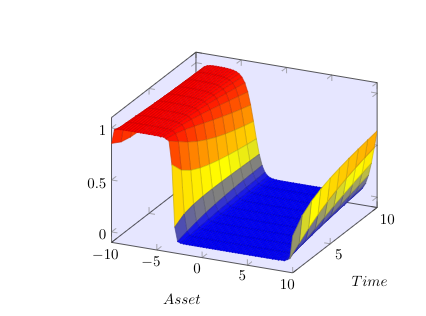}

\caption{Probability of default based on the CGMY Process for state 1.}
\label{fig_CGMY_S1}
\end{figure}




In order to study the numerical error of the meshfree RBF collocation method for evaluating the probability of default with respect to discretization parameters of $\Delta x$ and $\Delta \tau$, we calculate the log relative error obtained by 
\begin{equation*}
    \|E\|_{\text{LRE}} = \log\Big|\frac{\text{PD}_{\text{RBF}}-\text{PD}_{\text{FFT}}}{\text{PD}_{\text{FFT}}}\Big|
\end{equation*}
In Figure \ref{orderAxa}, we use this error metric to assess the rate of convergence of our RBF collocation method for evaluating the probability of default for the Axa case. It is evident from this figure that the rate of convergence of the numerical scheme is almost linear in time and super-linear in space. In general, it seems that the integral term related to the synchronous jump term (\ref{operator1}) the accuracy, and deteriorates the rate of convergence in time and specially space. 

Based on the fact that there is no real world example of calibrating the asset of a firm in tempered stable L\'evy model such as CGMY and Kobol models in the form of a regime-switching synchronous jump L\'evy process, so we have designed two artificial test cases based on these two processes similar to the example designed in \cite{bastani2013radial}.\\
\\
{\bf Test Problem II (CGMY Model)} 
For this example, we assume that the value of a hypothetical firm follows a CGMY L\'evy process in the context of regime-switching synchronise-jump dynamics containing 5 different regimes. Let us also assume that the  following generator matrix, $Q$, for the transition probability matrix is given by
\begin{equation}
Q = \begin{bmatrix}
-1 & 0.25 & 0.25 & 0.25 & 0.25\\
0.25 & -1 & 0.25 & 0.25 & 0.25\\
0.25 & 0.25 & -1 & 0.25 & 0.25\\
0.25 & 0.25 & 0.25 & -1 & 0.25\\
0.25 & 0.25 & 0.25 & 0.25 & -1\\
\end{bmatrix} 
\end{equation}
Data related to this model for five different states is provided in Table \ref{DataCGMY}.
\begin{table}[ht]
        \centering
        \caption{Parameters of the CGMY model.}
        \label{DataCGMY}
         \begin{tabular}{c|c|c|c|c|c} \hline
           CGMY Parameters & State 1 & State 2 & State 3 & State 4 & State 5 \\
           \hline
             $C$    &  0.1   & 0.3  & 0.5 & 0.7 & 0.9\\
             $G$    &  2.0  & 4.0 & 6.0 & 8.0 & 10.0  \\
             $M$   &  1.0  & 3.0 & 5.0 & 7.0 & 9.0 \\
             $Y$   &  0.11  & 0.22 & 0.33 & 0.44 & 0.55\\
             $\eta_{i,j}$ &0.01 & -0.02 & 0.03 & -0.04 & 0.05 \\
         
        \hline
         \end{tabular}
\end{table}
We have used a computational domain of the form, $[x_{\rm min}, x_{\rm max}] = [-10,10]$ for our collocation  method and also for truncating the integration domain. The approximated probability of default for this case and for time periods from 1 to 10 years is depicted in Figure (\ref{PDCGMY}). It must be noted that besides the Gaussian basis function, the cubic radial basis function, $\phi(r)=r^3$, works perfectly well and produces reliable and stable solutions. Figure (\ref{fig_CGMY_S1}) portraits the portability of default surface in this problem for the first regime. \\
\\
{\bf Test Problem III (KoBoL Model)} 
This test case is dedicated to finding the probability of default, when the dynamics of the firm's asset follows a regime-switching L\'evy process with three different regimes based on KoBoL jump process and  synchronous jumps effect. The generator matrix, $Q$ of the transition probability matrix is given by 
\begin{equation}
Q = \begin{bmatrix}
-1 & 0.3 & 0.7  \\
0.3 & -1 & 0.7 \\
0.7 & 0.3 & -1 
\end{bmatrix}, 
\end{equation}
with model parameters provided in Table (\ref{DataKobol}).
\begin{table}[ht]
        \centering
        \caption{Parameters of the KoBol model.}
        \label{DataKobol}
      \begin{tabular}{c|c|c|c} \hline %
           KoBol Parameters & State 1 & State 2 & State 3 \\
           \hline
             $C$    &  0.1  & 0.11  & 0.13\\
             $Y$    &  0.9  & 1.2   & 1.8\\
             $P$    &  0.2  & 0.4   & 0.8\\
             $q$    &  0.3  & 0.5   & 0.7\\
             $\lambda$&3.0  & 2     & 2.5\\ 
             $\eta_{i,j}$ &0.04 & -0.01 & 0.02 \\
             \hline
        \end{tabular}
\end{table}
Computational domain is $[x_{\rm min}, x_{\rm max}] = [-10,10]$ for collocation points. The probability of default of the firm for three different regimes is illustrated in Figure (\ref{KObol3State}) in the time horizon from 1 to 10 years when the default is happening just at the maturity. 
\begin{figure}[!ht]
\centering
%

\includegraphics[width=10cm]{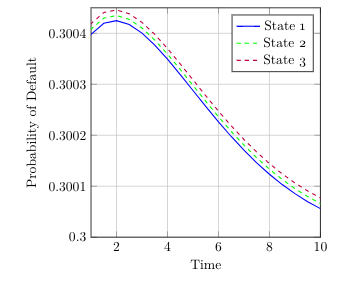}

\caption{Default probability for KoBol model.}
\label{KObol3State}
\end{figure}
Moreover, Figure (\ref{kobolSurfaces}) depicts the  probability of default surface in this case with three different regimes with the KoBol jump model over the space-time domain of $S\times t\in[e^{-10},e^{10}]\times [0,10]$. 


\begin{figure}[ht]

\begin{subfigure}{0.5\textwidth}

\includegraphics[width=7.5cm]{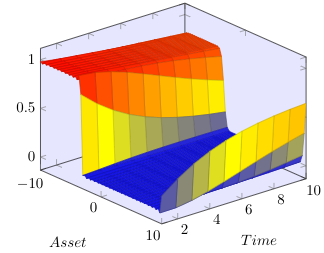}

\caption{}
\label{fig_STEM_S1}
\end{subfigure}
\begin{subfigure}{0.5\textwidth}

\includegraphics[width=7.5cm]{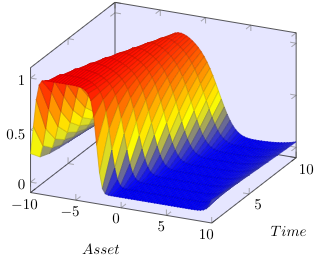}

\caption{}
\label{fig_STEM_S2}
\end{subfigure}

\centering
\begin{subfigure}{0.5\textwidth}

\includegraphics[width=7.5cm]{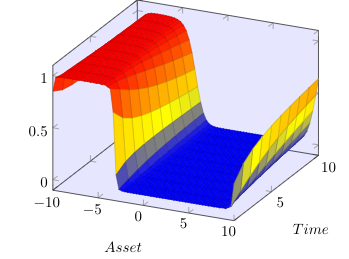}

\caption{}
\label{fig_STEM_S3}
\end{subfigure}
\caption{Probability of default in KoBol Model: (a) First regime; (b) Second regime; (c) Third regime}.
\label{kobolSurfaces}
\end{figure}


The stable behavior of these surfaces confirm that the proposed method is reliable and useful for problems with tempered stable jump terms as well as problems with more economic regimes.
\section{Conclusion}
In this investigation, a radial basis function collocation method is proposed to evaluate the default probability of a firm where the value of the firm is satisfying the newly proposed synchronous-jump regime-switching model. We explained the detailed methodology of the proposed method including deriving a semi-discrete system from discretizing the time derivative with finite-differences and meshfree collocation on the spatial direction. Generalizing the proposed synchronous-jump regime-switching model to benefit from tempered stable processes, we theoretically proved that the proposed numerical method is stable and inherently de-singularizes the problem. The numerical experiments confirm the efficiency and accuracy of the method. This method can be utilized for other structural models based on the synchronous-jump regime-switching model both for pricing financial derivatives and also credit risk computations. 

\FloatBarrier
\bibliography{main}
\end{document}